\documentclass[12pt,a4paper,oneside]{amsart}
\usepackage{amsfonts, amsmath, amssymb, amsthm}
\usepackage[colorlinks=true,citecolor=blue]{hyperref}
\usepackage[margin=1.4in]{geometry}

\usepackage{txfonts}
\usepackage[T1]{fontenc}

\usepackage{bbm}

\usepackage{mathtools}

\usepackage{graphicx}
\usepackage{parskip}
\usepackage{enumerate}
\usepackage{verbatim}
\usepackage{tikz}
\usepackage{algorithmic}

\usetikzlibrary{decorations,decorations.pathreplacing}

\begingroup
    \makeatletter
    \@for\theoremstyle:=definition,remark,plain\do{%
        \expandafter\g@addto@macro\csname th@\theoremstyle\endcsname{%
            \addtolength\thm@preskip\parskip
            }%
        }
\endgroup

\newtheorem{theorem}{Theorem}[section]
\newtheorem*{theorem*}{Theorem}
\newtheorem{lemma}[theorem]{Lemma}

\theoremstyle{definition}

\newtheorem*{remark*}{Remark}

\newtheorem*{claim*}{Claim}
\newtheorem{example}[theorem]{Example}
\newtheorem{problem}[theorem]{Problem}
\newtheorem{obs}[theorem]{Observation}

\newcommand{\df}[1]{{{\color{blue!50!black}\em #1}}}

\begin{document} 

\title{A Note on the Colorful fractional Helly theorem}  

\author{Minki Kim}
\date{\today}

\address{Minki Kim,\hfill \hfill \linebreak 
Department of Mathematical Sciences, \hfill \hfill \linebreak
KAIST, 
Daejeon, South Korea.  \hfill \hfill }
\email{kmk90@kaist.ac.kr}
\maketitle 

\begin{abstract}
	Helly's theorem is a classical result concerning the intersection patterns of convex sets in $\mathbb{R}^d$.
	Two important generalizations are the colorful version and the fractional version.
	Recently, B\'{a}r\'{a}ny et al. combined the two, obtaining a colorful fractional Helly theorem.
	In this paper, we give an improved version of their result.
\end{abstract}

\section{Introduction}
	Helly's theorem is one of the most well-known and fundamental results in combinatorial geometry, which has various generalizations and applications.
	It was first proved by Helly~\cite{Hel23} in $1913$, but his proof was not published until $1923$, after alternative proofs by Radon~\cite{Rad21} and K\"{o}nig~\cite{Kng22}.
	Recall that a family is \df{intersecting} if the intersection of all members is non-empty.
	The following is the original version of Helly's theorem.
\begin{theorem}[Helly's theorem, Helly {\cite{Hel23}}] \label{thm:hel}
	Let $\mathcal{F}$ be a finite family of convex sets in $\mathbb{R}^d$ with $|\mathcal{F}|\geq d+1$.
	Suppose every $(d+1)$-tuple of $\mathcal{F}$ is intersecting.
	Then the whole family $\mathcal{F}$ is intersecting.
\end{theorem}

	The following variant of Helly's theorem was found by Lov\'{a}sz, whose proof appeared first in B\'{a}r\'{a}ny's paper~\cite{Bar82}.
	Note that the original Helly's theorem is obtained by setting $\mathcal{F}_1=\mathcal{F}_2=\cdots=\mathcal{F}_{d+1}$.
\begin{theorem}[Colorful Helly theorem, Lov\'{a}sz {\cite{Bar82}}] \label{thm:colhel}
	Let $\mathcal{F}_1,\mathcal{F}_2,\dots,\mathcal{F}_{d+1}$ be finite, non-empty families (color classes) of convex sets in $\mathbb{R}^d$ such that every colorful $(d+1)$-tuple is intersecting.
	Then for some $1\leq i\leq d+1$, the whole family $\mathcal{F}_i$ is intersecting.
\end{theorem}

	One way to generalize Helly's theorem is by weakening the assumption: not necessarily all but only a positive fraction of $(d+1)$-tuples are intersecting.
	The following theorem shows how the conclusion changes.
\begin{theorem}[Fractional Helly theorem, Katchalski and Liu {\cite{KL79}}] \label{thm:frachel}
	For every $\alpha\in(0,1]$, there exists $\beta=\beta(\alpha,d)\in(0,1]$ such that the following holds:
	Let $\mathcal{F}$ be a finite family of convex sets in $\mathbb{R}^d$ with $|\mathcal{F}|\geq d+1$.
	If at least $\alpha{|\mathcal{F}|\choose d+1}$ of the $(d+1)$-tuples in $\mathcal{F}$ are intersecting,
	then $\mathcal{F}$ contains an intersecting subfamily of size at least $\beta|\mathcal{F}|$.
\end{theorem}
	The fractional variant of Helly's theorem first appeared as a conjecture on interval graphs, i.e. intersection graphs of families of intervals on $\mathbb{R}$.
	Abbott and Katchalski~\cite{AK79} proved that $\beta=1-\sqrt{1-\alpha}$ is optimal for every family whose intersection graph is a chordal graph.
	Note that, by the result of Gavril~\cite{Gav74}, interval graphs are chordal graphs.
	
	The fractional Helly theorem for arbitrary dimensions was proved by Katchalski and Liu~\cite{KL79}.
	Their proof gives a lower bound $\beta\geq\alpha/(d+1)$.
	However, it seems natural that $\beta$ tends to $1$ as $\alpha$ tends to $1$, since the original Helly's theorem implies that $\beta=1$ when $\alpha=1$.
	Later, the quantitatively sharp value $\beta(\alpha,d)=1-(1-\alpha)^{1/(d+1)}$ was found by Kalai~\cite{Kal84} and Eckhoff~\cite{Eck85}.
	The result follows from the upper bound theorem for families of convex sets.

	The $(p,q)$-theorem, another important generalization of Helly's theorem, deals with a weaker version of the assumption, the so-called $(p,q)$-condition: for every $p$ members in a given family, there are some $q$ members of the family that are intersecting.
	For instance, the $(d+1,d+1)$-condition in $\mathbb{R}^d$ is the hypothesis of Helly's theorem.
	The $(p,q)$-theorem was proved by Alon and Kleitman~\cite{AKl92}, settling a conjecture by Hadwiger and Debrunner~\cite{HD57}.

	The proof of the $(p,q)$-theorem is quite long and involved, using various techniques.
	However, one of the most crucial ingredients is the fractional Helly theorem.
	See the survey paper by Eckhoff~\cite{Eck93} and the textbook by Matousek~\cite{Mat02} for an overview and further knowledge of this field.
	
	Recently, B\'{a}r\'{a}ny et al.~\cite{BFMOP14} established the colorful and fractional versions of the $(p,q)$-theorem.
	A key ingredient in their proof was a colorful variant of the fractional Helly theorem.
\begin{theorem}[B\'{a}r\'{a}ny-Fodor-Montejano-Oliveros-P\'{o}r {\cite{BFMOP14}}] \label{thm:bfmop}
	Let $\mathcal{F}_1,\mathcal{F}_2,\dots,\mathcal{F}_{d+1}$ be finite, non-empty families (color classes) of convex sets in $\mathbb{R}^d$, and assume that $\alpha\in(0,1]$.
	If at least $\alpha|\mathcal{F}_1|\cdots|\mathcal{F}_{d+1}|$ of the colorful $(d+1)$-tuples are intersecting, then some $\mathcal{F}_i$ contains an intersecting subfamily of size $\frac{\alpha}{d+1}|\mathcal{F}_i|$.
\end{theorem}
	
	Note that for $\alpha=1$ we recover the hypothesis of the colorful Helly theorem, and we should therefore expect the value $\beta =1$ (rather than $\beta = \frac{1}{d+1}$).
	B\'{a}r\'{a}ny et al. therefore proposed the problem of showing that the function $\beta$ in Theorem~\ref{thm:bfmop} tends to $1$ as $\alpha$ tends to $1$.

	In this paper, we solve the problem of  B\'{a}r\'{a}ny et al.
\begin{theorem} \label{thm:cfhel}
	For every $\alpha\in(0,1]$, there exists $\beta=\beta(\alpha,d)\in(0,1]$ tending to $1$ as $\alpha$ tends to $1$ such that the following holds:
	Let $\mathcal{F}_1,\mathcal{F}_2,\dots,\mathcal{F}_{d+1}$ be finite, non-empty families (color classes) of convex sets in $\mathbb{R}^d$.
	If at least $\alpha|\mathcal{F}_1|\cdots|\mathcal{F}_{d+1}|$ of the colorful $(d+1)$-tuples are intersecting, then for some $1\leq i\leq d+1$,
	$\mathcal{F}_i$ contains an intersecting subfamily of size $\beta|\mathcal{F}_i|$.
\end{theorem}

See the survey paper by Amenta, Loera, and Sober\'{o}n~\cite{ALS15} for an overview of recent results and open problems related to Helly's theorem.

\bigskip


\section{Proof of Theorem~\ref{thm:cfhel}}\label{sec:pf}

\smallskip
\subsection{The Matching number of hypergraphs}\label{subsec:tau}
	Let $\mathcal{H}$ be an $r$-uniform hypergraph on a vertex set $X$.
	A subset $S\subseteq X$ is said to be an \df{independent set} in $\mathcal{H}$ if the induced sub-hypergraph $\mathcal{H}[S]$ contains no hyperedge.
	The \df{independence number} $\alpha(\mathcal{H})$ of $\mathcal{H}$ is the cardinality of a maximum independent set in $\mathcal{H}$.
	A \df{matching} of $\mathcal{H}$ is a set of pairwise disjoint edges in $\mathcal{H}$.
	The \df{matching number} $\nu(\mathcal{H})$ of $\mathcal{H}$ is the cardinality of a maximum matching in $\mathcal{H}$.
	For our result, we need the following observation.
\begin{obs}\label{lem:nu}
	Let $\mathcal{H}=(X,E)$ be an $r$-uniform hypergraph with $|X|=n$.
	Suppose $\alpha(\mathcal{H})<cn$ for some $c\in(0,1]$.
	Let $M$ be a maximum matching in $\mathcal{H}$.
	Note that $X\setminus M$ is an independent set it $\mathcal{H}$.
	If not, assume that there is an edge $e$ contained in $X\setminus M$.
	Then $M\cup\{e\}$ is a matching in $\mathcal{H}$, which is a contradiction to the maximality of $M$.
	Thus $|X\setminus M|=n-r\nu(\mathcal{H})\leq \alpha(\mathcal{H})<cn$, so $\nu(\mathcal{H})>\frac{n-cn}{r}$.
\end{obs}
\smallskip
\subsection{Proof of Theorem~\ref{thm:cfhel}}
	Theorem 1.6 is implied by the following more explicit result.
\begin{theorem}\label{claim:cfhel}
	For every $\alpha\in(0,1]$, the following holds: 
	Let $\mathcal{F}_1,\mathcal{F}_2,\dots,\mathcal{F}_{d+1}$ be finite families (color classes) of convex sets in $\mathbb{R}^d$.
	If at least $\alpha|\mathcal{F}_1|\cdots|\mathcal{F}_{d+1}|$ of the colorful $(d+1)$-tuples are intersecting, then for some $1\leq i\leq d+1$,
	$\mathcal{F}_i$ contains an intersecting subfamily of size at least
	\[\max\left\{\frac{\alpha}{d+1}, 1-(d+1)(1-\alpha)^{\frac{1}{d+1}}\right\}|\mathcal{F}_i|.\]
\end{theorem}

The following is a key lemma of the proof of Theorem~\ref{claim:cfhel}.
\begin{lemma} \label{obs:colhel}
	Choose any $(d+1)$ members from each color class, say $\mathcal{F}'_1,\dots,\mathcal{F}'_{d+1}$.
	If each of $\mathcal{F}'_i$ is not intersecting, then at least one of colorful $(d+1)$-tuple is not intersecting.
\end{lemma}
\begin{proof}
This follows directly from the colorful Helly theorem.
\end{proof}

\begin{proof}[Proof of Theorem~\ref{claim:cfhel}]
	It is sufficient to show that for every $\alpha\in[1-\frac{1}{(d+1)^{(d+1)}},1]$, 
	if at least $\alpha|\mathcal{F}_1|\cdots|\mathcal{F}_{d+1}|$ of the colorful $(d+1)$-tuples are intersecting, then some $\mathcal{F}_i$ contains an intersecting subfamily of size at least $(1-(d+1)(1-\alpha)^{\frac{1}{d+1}})|\mathcal{F}_i|$.

	Let $\mathcal{F}$ be the disjoint union of $\mathcal{F}_1,\mathcal{F}_2,\dots,\mathcal{F}_{d+1}$.
	For each $1\leq i\leq d+1$, denote $n_i=|\mathcal{F}_i|$ and define a $(d+1)$-uniform hypergraph $H_i:=(\mathcal{F}_i,E_i)$ whose vertices are the members in $\mathcal{F}_i$ and hyperedges are non-intersecting $(d+1)$-tuples in $\mathcal{F}_i$.
	Let $\nu_j=\nu(H_j)$ for each $1\leq j\leq d+1$.
	
	Also define a $(d+1)$-uniform hypergraph $H:=(\mathcal{F},E)$ whose vertices are the members in $\mathcal{F}$ and hyperedges are intersecting colorful $(d+1)$-tuples in $\mathcal{F}$.
	
	Given $\alpha\in[1-\frac{1}{(d+1)^{(d+1)}},1]$, let $\gamma=\gamma(\alpha,d)=1-(d+1)(1-\alpha)^{\frac{1}{d+1}}$.
	For contraction, assume that in each family $\mathcal{F}_j$, every subfamily of size at least $\gamma n_j$ has an empty intersection.

	By Lemma~\ref{obs:colhel}, we have $\alpha n_1\cdots n_{d+1}\leq|E|\leq n_1\cdots n_d+1 - \nu_1\cdots \nu_{d+1}$.
	Recall that by Observation~\ref{lem:nu}, $\nu_j>\frac{n_j-\gamma n_j}{d+1}=\left(\frac{1-\gamma}{d+1}\right)n_j$ for each $1\leq j\leq d+1$.
	Then we obtain
	\begin{eqnarray*}
	\alpha n_1\cdots n_{d+1} &\leq& n_1\cdots n_{d+1}-\nu_1\cdots\nu_{d+1}\\
	&<& n_1\cdots n_{d+1}-\left(\frac{1-\gamma}{d+1}\right)^{d+1}n_1\cdots n_{d+1}\\
	&=&\left(1-\left(\frac{1-\gamma}{d+1}\right)^{d+1}\right)n_1\cdots n_{d+1},
	\end{eqnarray*}
	hence $\alpha<1-(\frac{1-\gamma}{d+1})^{d+1}=\alpha$, which is a contradiction.
	
	Thus, there should exist $1\leq i\leq d+1$ such that $\mathcal{F}_i$ contains an intersecting subfamily of size $(1-(d+1)(1-\alpha)^{\frac{1}{d+1}})n_i$. 
\end{proof}

\bigskip
\section{The Upper bound}\label{sec:uppbd}
	First recall that in the fractional Helly theorem, the upper bound is given by \[\beta=\beta(\alpha,d)\leq (1-(1-\alpha)^{\frac{1}{d+1}}).\]
	This can be seen by the following well-known construction, which also shows the exactness of upper bound theorem for convex sets~\cite{Eck85}\cite{Kal84}.
\begin{example}
	Let $\mathcal{F}$ consist of $\lfloor\beta n\rfloor-(d+1)$ copies of $\mathbb{R}^d$ and $n-\lfloor\beta n\rfloor+(d+1)$ hyperplanes in general position.
	Denote by $f_d(\mathcal{F})$ the number of intersecting $(d+1)$-tuples in $\mathcal{F}$.
	Note that
	\begin{eqnarray*}
	\alpha{n\choose d+1}=f_d(\mathcal{F})&=&{n\choose d+1}-{n-(\lfloor\beta n\rfloor-(d+1))\choose d+1}\\
	&<&{n\choose d+1}-{n-\lfloor\beta n\rfloor\choose d+1}\\
	&\leq&{n\choose d+1}-(1-\beta)^{d+1}{n\choose d+1}.
	\end{eqnarray*}
\end{example}

The colorful version of this example gives an upper bound for the colorful fractional Helly theorem.
\begin{theorem}\label{thm:uppbd}
	For every $\alpha\in(0,1]$, there exist finite families (color classes) $\mathcal{F}_1,\dots,\mathcal{F}_{d+1}$ of convex sets in $\mathbb{R}^d$ such that the following holds.
	$\alpha|\mathcal{F}_1|\cdots|\mathcal{F}_{d+1}|$ of the colorful $(d+1)$-tuples are intersecting, but in each color class $\mathcal{F}_i$, the maximum cardinality of an intersecting subfamily is at most $(1-(1-\alpha)^{\frac{1}{d+1}})|\mathcal{F}_i|$.
\end{theorem}
\begin{proof}
	It follows from the following construction.
	Let $\mathcal{F}_i$ consist of $\lfloor\beta n\rfloor-d$ copies of $\mathbb{R}^d$ and $n-\lfloor\beta n\rfloor+d$ hyperplanes in general position.
	Moreover, let all hyperplanes in $\mathcal{F}_1\cup\cdots\cup\mathcal{F}_{d+1}$ be in general position.
	Note that each of $\mathcal{F}_i$ has an intersecting subfamily of size at most $\beta n$.
	The number of colorful $(d+1)$-tuples is given by
	\[\alpha n^{d+1}=n^{d+1}-(n-\lfloor\beta n\rfloor+d)^{d+1}<n^{d+1}-(1-\beta)^{d+1}n^{d+1}.\]
	As $n$ tends to infinity, one may have that $\alpha = 1-(1-\beta)^{d+1}$, i.e. $\beta=1-(1-\alpha)^{\frac{1}{d+1}}$.
\end{proof}

\bigskip
\section{Remarks}\label{sec:remark}

	In this note, we found upper and lower bounds on the function $\beta(\alpha,d)$ in the colorful fractional Helly theorem, however, there remains a large gap between them.
	It would be interesting to determine the exact value of $\beta(\alpha,d)$.
\begin{problem}~\label{prob:1}
	What is the exact value of $\beta=\beta(\alpha,d)$ in Theorem~\ref{thm:cfhel}?
\end{problem}
	It is easy to see that $\beta(\alpha,1)=1-\sqrt{1-\alpha}$ is the optimal bound for $d=1$.
	We conjecture that $\beta(\alpha,d)=1-(1-\alpha)^{\frac{1}{d+1}}$ is the optimal bound for $d>1$.
\bigskip
\section*{Acknowledgment}
I would like to thank my advisor Prof.Andreas Holmsen for introducing the problem and providing insightful comments. I also would like to thank Dr.Ilkyoo Choi for his advice to improve the readability of the paper.

\bibliographystyle{abbrv}
\bibliography{bib-matroid}

\end{document}